\documentclass[a4paper,10pt]{amsart}

\usepackage[english]{babel}
\usepackage{amssymb}
\usepackage[sort]{natbib}
\usepackage{mathrsfs}
\usepackage{comment}
\usepackage{mathscinet}
\usepackage{fullpage}
\usepackage[colorlinks,linkcolor=blue,citecolor=blue]{hyperref}
\usepackage{color}
\usepackage{kamil}
\definecolor{c20}{rgb}{0.,0.7,0.}
\definecolor{c30}{rgb}{0.,0.,1.}
\definecolor{c40}{rgb}{1,0.1,0.7}
\definecolor{c50}{rgb}{1,0,0}
\def\rd#1{\textcolor{black}{#1}}

\begin{document}
\bibliographystyle{plainnat}
\setcitestyle{numbers}

\title{Sample path properties of reflected Gaussian processes}

\author{Kamil Marcin Kosi\'nski}
\address{Mathematical Institute, University of
Wroc\l aw, pl.\ Grunwaldzki 2/4, 50-384 Wroc\l aw, Poland.}
\email{Kamil.Kosinski@math.uni.wroc.pl}
\author{Peng Liu}
\address{
Department of Actuarial Science, University of Lausanne, UNIL-Dorigny 1015 Lausanne, Switzerland}
\email{Peng.Liu@unil.ch}



\date{\today}

\begin{abstract}
We consider a stationary queueing process $Q_X$ fed by a centered Gaussian process $X$ with stationary increments and variance function satisfying classical regularity conditions. A criterion when, for a given function $f$, $\mathbb P (Q_{X}(t) > f(t)\, \text{ i.o.})$  equals 0 or 1 is provided. Furthermore, an Erd\"os--R\'ev\'esz type law of the iterated logarithm is proven for the last passage time $\xi (t) = \sup\{s:0\le s\le t, Q_{X}(s)\ge f(s)\}$. Both of these findings extend previously known results that were only available for the case when $X$ is a fractional Brownian motion.
\end{abstract}

\maketitle

\section{Introduction and Main Results}
\label{sec:intro}
Let $X=\{X(t):t\ge 0\}$ be a centered Gaussian process with stationary increments and almost surely continuous sample paths. Given $c>0$, consider a \textit{reflected} (at 0) Gaussian process $Q_X = \{Q_X(t):t\ge0\}$ given by the following formula
\begin{equation}
\label{prob.0}
Q_X(t)=X(t)-ct+\max\left(Q_X(0),-\inf_{s\in[0,t]}(X(s)-cs)   \right).
\end{equation}
It is well known in queueing and risk theory, e.g., \cite{Reich58}, that the unique stationary solution of \eqref{prob.0} has the following representation
\[
Q_X(t)=\sup_{-\infty< s\le t}\left(X(t)-X(s)-c(t-s)\right).
\]
Due to numerous application, $Q_X$ has been studied in the literature under different levels of generality,
e.g., \cite{Norros94,Husler99,Debicki02,Husler04b,Dieker05a,Hashorva13,Liu15}.

Let \rd{$f$ be any positive nondecreasing function on $\rr$}. Kolmogorov's zero-one law implies that the process $Q_X$ crosses the function $f$ infinitely many times with probability 0 or 1. Assume that $\prob{Q_X(t)>f(t)\text{ i.o.}}=1$ and define $\xi_{f}=\{\xi_f(t):t\ge 0\}$ as the last crossing time before time $t$, that is,
\[
\xi_f (t) = \sup\{s:0\le s\le t, Q_{X}(s)\ge f(s)\}.
\]
By the assumption on $f$ it follows that
\[
\lim_{t\to\infty}\xi_f(t)=\infty\quad\text{ and }\quad\limsup_{t\to\infty}(\xi_f(t)-t)=0\quad \text{a.s.}
\]
The purpose of this paper is to provide a tractable criterion to verify the zero-one law as well as to give the asymptotic lower bound on $\xi_f(t)-t$. \citet{Erdos90} investigated the lower bound in the case when $Q_X$ is substituted by Brownian motion $W$ and $f(t) = \sqrt{2t\log_2 t}$ with $\log_2 t = \log\log t$. Subsequently similar results are known as Erd\"os--R\'ev\'esz type law of the iterated logarithm.

In the reminder of the paper we impose the following assumptions on variance function $\sigma^2$ \rd{of $X$}:
\begin{itemize}
\item[\bf AI:] \rd{$\lim_{t\toi}\sigma^2(t)/t^{2\alpha_\infty} =A_\infty$, for some} $A_\infty>0$, $\alpha_\infty\in (0,1)$. Further, \rd{$\sigma^2$ is positive} and twice continuously differentiable on $(0,\infty)$ with its first derivative $\dot{\sigma^2}$ and second derivative $\ddot{\sigma^2}$ being ultimately monotone at $\infty$.
\item[\bf AII:] \rd{$\lim_{t\to0^+}\sigma^2(t)/t^{2\alpha_0} = A_0$, for some} $A_0>0$, $\alpha_0\in(0,1]$.
\end{itemize}
Assumptions \textbf{AI-AII} allow us to cover models that play important role in
Gaussian storage models, including both
aggregations of \textit{fractional Brownian motions} and \textit{integrated stationary Gaussian processes};
see, e.g., \cite{Norros94,Husler99,Dieker05a,Debicki02}. In further analysis we tacitly assume that the variance function $\sigma^2$ of $X$ satisfies both {\bf AI} and {\bf AII}. Our first contribution is the following criterion; see, e.g., \cite{Watanabe70,Qualls71} for similar results in the classical setting of \textit{non-reflected} stationary Gaussian process.
\begin{theorem}
\label{thm:equiv}
For all positive and nondecreasing functions \rd{$f$ on some interval $[T,\infty)$, $T>0$},
\[
\prob{Q_{X}(t) > f(t)\quad\text{\rm{i.o.}}} = 0\quad \text{or}\quad 1,
\]
according as the integral
\[
\int_T^\infty
\frac{\psi(f(u))}{f(u)}
\D u\quad \text{is finite or infinite},
\]
where
\[
\psi(u):=\prob{\sup_{t\in[0,u]} Q_{X}(t) > u}.
\]
\end{theorem}

With $\overleftarrow{m}$ being the generalized inverse of
\[
m(u)=\inf_{t\geq 0}\frac{u(1+ct)}{\sigma(ut)},
\]
define function $f_p$ by
\begin{equation}
\label{eq:fp&gamma}
f_p(t) = \overleftarrow{m}\left(\sqrt{2\left(\log t+\left(\frac{\gamma-1}{2(1-\alpha_\infty)}-p\right)\log_2 t\right)}\right),
\quad
\gamma=\left\{\begin{array}{cc}
\frac{2(1-\alpha_\infty)}{\alpha_\infty}& \alpha_\infty\geq 1/2\\
\frac{2(1+\alpha_0-2\alpha_\infty)}{\alpha_0}& \alpha_\infty<1/2
\end{array}\right.,
\end{equation}
and a positive constant $\mathscr C$ as
\[
\mathscr C =  \half (\mathcal{H}_{\eta_{\alpha_\infty}})^2\sqrt{\frac{A}{B}}\zeta_{\alpha_\infty}
\left(
\frac{\sqrt{2 A_\infty}}{A}
\right)^{\frac{\gamma-1}{1-\alpha_\infty}},
\]
where the remaining constants are defined in \autoref{sec:prelim}. Since the exact asymptotics of $\psi(u)$, as $u$ grows large,
were found in \citep{Debicki16}, c.f., \autoref{lem:asymptotics}, it follows that
\begin{equation}
\label{eq:fp&CbyEq}
\frac{\psi(f_p(u))}{f_p(u)}
=\mathscr C (u\log^{1-p} u)^{-1}(1+o(u)),\as u.
\end{equation}
Hence, by \autoref{thm:equiv}, $\prob{Q_{B_H}(t) > f_p(t)\text{ i.o.}} = 1$ provided that $p\ge 0$, which leads to the following conclusion after deriving the exact asymptotics of $f_p$.
\begin{corollary}
\label{cor:main}
\[
\limsup_{t\toi} \frac{Q_{X}(t)}
                     {(\log t)^{\frac{1}{2(1-\alpha_\infty)}}} =\left(\frac{2A_\infty}{A^2}\right)^{\frac{1}{2(1-\alpha_\infty)}}
\quad\text{a.s.}
\]
\end{corollary}

Our second contribution is as follows.
\begin{theorem}
\label{thm:main}
If $p>1$, then
\[
\liminf_{t\toi}\frac{\xi_{f_p}(t)-t}{h_p(t)} = - 1\ \ {\rm a.s.}
\]
If $p\in(0,1]$, then
\[
\liminf_{t\toi}\frac{\log\left(\xi_{f_p}(t)/t\right)}{h_p(t)/t } = - 1\ \ {\rm a.s.},
\]
where
\[
h_p(t) = p\frac{f_p(t)}{\psi(f_p(t))}\log_2 t.
\]
\end{theorem}
\autoref{thm:main} shows that for $t$ big enough,
there exists an $s$ in $[t - h_p(t), t]$  such that
$Q_{X}(s)\ge f_p(s)$ and that the length of the interval
$h_p(t)$ is smallest possible. \autoref{thm:equiv} and \autoref{thm:main} generalize the main results of \cite{Debicki17a}, which considered the special case when $X\equiv B_H$ is a fractional Brownian motion with any Hurst parameter $H\in(0,1)$; see also \citep{Shao92,Debicki17} for similar results for non-reflected Gaussian processes and Gaussian order statistics.
The organization of the rest of paper is as follows. The notation and examples of Gaussian processes $X$ that fall under our framework are displayed in \autoref{sec:notation} followed by
properties of the storage process $Q_X$ in \autoref{sec:prelim}. \autoref{sec:auxLem} gives two useful tools and some auxiliary lemmas for the proof of the main results which are presented in \autoref{sec:Proofs}.

\section{Notation and Special Cases}
\label{sec:notation}
We write $f(u)\sim g(u)$ if $\lim_{u\to \infty}f(u)/g(u)=1$.
By \rd{$\overleftarrow{\sigma}$ we denote the generalized inverse function to
${\sigma}$},
\rd{$\Psi$ denotes the tail distribution function} of the standard Normal random variable.
Function $f$ is ultimately monotone if there exists a constant $M>0$ such that
\rd{$f$} is monotone over $(M,\infty)$.
For a centered continuous Gaussian process \rd{with stationary increments}
$V=\{V(t): t\in\rr\}$, such that $V(0)=0$,
\begin{equation}
\label{sv}
\Cov(V({t}), V({s}))=\frac{\sigma_V^2({t})+\sigma_V^2({s})-\sigma_V^2({t}-{s})}{2},
\end{equation}
we introduce \textit{the generalized Pickands' constant} on a compact set $E\subset\rr^d$ as
\[
\mathcal{H}_V(E)=\ee\exp\left(\sup_{ t\in E}\left(\sqrt{2}V( t)-\sigma_{V}^2( t)\right)\right).
\]
Let
\[
 \mathcal{H}_V=\lim_{S\to\infty}\frac{\mathcal H_V([0,S])}{S}.
\]
We refer to \citep{Debicki14} for the finiteness of $\mathcal{H}_V(E)$ and to \citep{Debicki17b, Debicki17c} for the fact that $\mathcal H_V\in(0,\infty)$. Furthermore, see \citep{Debicki02,Dieker05a} for the analysis of other properties of Pickands'-type constants.

\subsection*{Special cases}
{\it \underline{Fractional Brownian motion}}. Let $B_H=\{B_H(t): t\geq 0\}$ denote fBm with Hurst index $H\in (0,1]$ which is a centered Gaussian processes with continuous sample paths and covariance function satisfying
\[
\Cov\left(B_H(t), B_H(s)\right)=\frac{|s|^{2H}+|t|^{2H}-|t-s|^{2H}}{2},\quad  s, t\geq 0.
\]
Direct calculations show that
\[
\sigma^2(t)=|t|^{2H}, \quad m(u)=Au^{1-H}, \quad A=\left(\frac{{H}}{c(1-{H})}\right)^{-H}\frac{1}{1-H}, \quad B=\left(\frac{H}{c(1-H)}\right)^{-H-2}H,\] \[ \overleftarrow{m}(u)=A^{-\frac{1}{1-H}}u^{\frac{1}{1-H}},\quad
f_p(u)=\left(\frac{2}{A^2}\left(\log u+\left(\frac{2-3H}{2H(1-H)}-p\right)\log_2 u\right)\right)
^{\frac{1}{2(1-H)}},
\]
\[
h_p(u)=p \mathscr C^{-1} u\log^{1-p} u \log_2u, \quad \mathscr C=\half (\mathcal{H}_{B_H})^2\sqrt{\frac{A}{B}}\left(\frac{\sqrt{2}(\tau^*)^{2H}}{1+c\tau^*}\right)
\left(
\frac{\sqrt{2}}{A}
\right)^{\frac{2-3H}{H(1-H)}},
\]
with $\tau^*=\frac{H}{c(1-H)}$.
This coincides with \citep[Theorem 1 and 2]{Debicki17a}.\\
{\it \underline{Short-range dependent Gaussian integrated processes}}. Let $X(t)=\int_0^t Y(s)\D s$ where $Y$ is a centered stationary Gaussian process with unit variance and correlation function $r(t)=\Cov(Y(s+t), Y(s)), s\geq 0, t\geq 0$. We say that $X$ possesses short-range dependence property if:
\begin{itemize}
\item[\textbf{S1}:] \rd{$r$ is a continuous function on $[0,\infty)$ such that, $\lim_{t\to \infty}tr(t)=0;$}
\item[\textbf{S2}:] \rd{$r$ is decreasing over $[0,\infty)$ and $\int_0^\infty r(t)\D t=\frac{1}{G}$ for some $0<G<\infty$};
\item[\textbf{S3}:] $\int_0^\infty s^2 |r(s)| \D s<\infty$.
\end{itemize}
The above assumptions go line by line the same as the assumptions in \cite{Debicki02} except a little modification. {\bf S1-S3} cover wide range of stationary Gaussian processes such as  the process with  correlation function
\[
r(t)=e^{-|t|^\alpha}, \quad \alpha\in (0,2].
\]
In particular if $r(t)=e^{-|t|}$, \rd{$X$} is the so-called Ornstein-Uhlenbeck process. Apparently, if {\bf S1-S3} are satisfied, then
\[
\sigma^2(t)=2\int_0^t\int_0^s r(v)\D v\D s\]
satisfies \textbf{AI-AII}. Note that
\[
\sigma^2(t)\sim t^2, \quad\text{as } t\to 0, \quad \sigma^2(t)\sim \frac{2}{G} t, \as t.
\]
\cite[Proposition 6.1]{Debicki02} shows that
\[
m^2(u)=2Gu+2G^2G_1+o(1), \as u,
\]
with \rd{$G_1=\int_0^\infty t r(t)\D t$}. This indicates that $m(u)$ can be replaced by $\widehat{m}(u)=\sqrt{2Gu+2G^2G_1}$ in \autoref{lem:asymptotics} and \autoref{thm:main}. Under this replacement, we have that
\[
\overleftarrow{\widehat{m}(u)}=\frac{u^2}{2G}-GG_1,\quad
 f_p(t)=\frac{1}{G}\left(\log t+\left(1-p\right)\log_2 t\right)-GG_1
\]
and
\[
h_p(u)=p \mathscr C^{-1} u\log^{1-p} u \log_2u, \quad \mathscr C=
\frac{2(\mathcal{H}_{\eta_{1/2}})^2\left(\overleftarrow{\sigma}(\frac{\sqrt{2}}{cG})\right)^{-2}}{A^{3/2}\sqrt{B}G},
\]
with $\eta_{1/2}=\frac{cG}{\sqrt{2}}X(\overleftarrow{\sigma}(\frac{\sqrt{2}}{cG})t), A=2c^{1/2},$ and $ B=\frac{1}{2}c^{5/2}$.

\section{Properties of the storage process}
\label{sec:prelim}
Before we present our auxiliary results, we need to introduce some notation and state some properties of the supremum of the  process $Q_{X}$ as derived in \citep{Piterbarg01, Husler04b}. We begin with the relation
\begin{equation}
\label{eq:basic_transformation}
\prob{\sup_{t\in[0,T]} Q_{X}(t)>u} =\prob{ \sup_{\substack{s\in[0,T/u] \\ \tau\ge0}} Z_u(s,\tau)>m(u)},\quad\text{for any}\quad T>0,
\end{equation}
where
\[
Z_{u}(s,\tau)= \frac{X(u(\tau+s))-X(us)}{u(1+c\tau)} m(u).
\]
Note that $Z_u(s,\tau)$ is a Gaussian field, stationary in $s$, but not in $\tau$.
The variance $\sigma_{u}^2(\tau)$ of $Z_u(s,\tau)$ equals $\frac{\sigma^2(u\tau)}{(u(1+c\tau))^2}m^2(u)$ and  $\sigma_{u}(\tau)$ has a single maximum point at $\tau(u)$
for $u$ sufficiently large with
$\lim_{u\to\infty}\tau(u)=\tau^*$, where
\begin{align}\label{tau}
\tau^* = \frac{\alpha_\infty}{c(1-\alpha_\infty)}.
\end{align}
\rd{Taylor's formula shows that, for each $u>0$ sufficiently large,
$$\sigma_u(\tau)=\sigma_u(\tau(u))+\dot{\sigma}_u(\tau(u))(\tau-\tau(u))+\frac{1}{2}\ddot{\sigma}_u(\xi)(\tau-\tau(u))^2$$
with $\xi\in (\tau,\tau(u))$. Noting that $\sigma_u(\tau(u))=1$,  for $u$ sufficiently large,
$\dot{\sigma}_u(\tau(u))=0$ and for  $\lim_{u\to\infty}\delta_u=0$
$$\lim_{u\to\infty}\sup_{|\tau-\tau(u)|<\delta_u}\left|\frac{1}{2}\ddot{\sigma}_u(\xi)-\frac{B}{2A}\right|=0,$$
we have }
\begin{align}
\label{variance}
\lim_{u\to\infty}\sup_{\tau\neq \tau(u), |\tau-\tau(u)|<\delta_u}
\left|\frac{1-\sigma_{u}(\tau)}{\frac{B}{2A}(\tau-\tau(u))^2}-1\right| =0
\end{align}
with $\lim_{u\to\infty}\delta_u=0$, where
\begin{align}\label{AB}
A=\left(\frac{{\alpha_\infty}}{c(1-{\alpha_\infty})}\right)^{-\alpha_\infty}\frac{1}{1-\alpha_\infty}, \quad
B=\left(\frac{\alpha_\infty}{c(1-\alpha_\infty)}\right)^{-\alpha_\infty-2}\alpha_\infty.
\end{align}

Let $r_{u,u'}(s,\tau,s',\tau')$ be the correlation function
of $Z_u(s,\tau)$ and $Z_{u'}(s', \tau')$. Then
\begin{align*}
r_{u,u'}&(s,\tau,s',\tau')\\
& =\frac{-\sigma^2(|us-u's' + u\tau -u'\tau'|)+\sigma^2(|us-u's'+u\tau|)+\sigma^2(|us - u's'-u'\tau'|)-\sigma^2(|us-u's'|)}{2\sigma(u\tau)\sigma(u'\tau')}.
\end{align*}
Denote by
$$r_u(s,\tau, s',\tau')=r_{u,u}(s,\tau,s',\tau').$$
Then Lemma 5.4 in \cite{Debicki16} gives that with $\delta_u>0$ and $\lim_{u\to\infty}\delta_u=0$
\begin{align}\label{cor0}
\lim_{u\to\infty}\sup_{(s,\tau)\neq (s',\tau'),|\tau-\tau(u)|, |\tau'-\tau(u)|, |s-s'|\leq \delta_u }\left|\frac{1-r_u(s,\tau, s',\tau')}{\frac{\sigma^2(u|s-s'+\tau-\tau'|)+\sigma^2(u|s-s'|)}{2\sigma^2(u\tau^*)}}-1\right|=0.
\end{align}

Now assume that
\begin{align}\label{res}\frac{u\tau +u'\tau'}{|us-u's'|}<\frac{1}{2},
\end{align}
 and without loss of generality, $us>u's'$. Then \rd{Taylor's formula} gives that
\[
r_{u,u'}(s,\tau,s',\tau')=\frac{-\ddot{\sigma^2}(|us-u's'+v_1-v_2|)u\tau u'\tau'}{2\sigma(u\tau)\sigma(u'\tau')},
\]
with $v_1\in (0,u\tau), v_2\in (0,u'\tau')$. 
Noting that by (\ref{res})
$$|us-u's'+v_1-v_2|\geq u\tau+ u'\tau, $$
 in light of \citep[Theorem 1.7.2]{Bingham87} and by \textbf{AI-AII} we have
\begin{align*}
\frac{|us-u's'+v_1-v_2|^2\ddot{\sigma^2}(|us-u's'+v_1-v_2|)}{\sigma^2(|us-u's'+v_1-v_2|)}\to 2\alpha_\infty(2\alpha_\infty-1), \as{u\tau, u'\tau'}.
\end{align*}
Hence
\begin{align}\label{ruu'}
r_{u,u'}(s,\tau,s',\tau')
\sim -\alpha_\infty(2\alpha_\infty-1)\left|\frac{\sqrt{u\tau u'\tau'}}{us-u's' +v_1-v_2}\right|^{2\lambda}
,\as{u\tau, u'\tau'},
\end{align}
where $\lambda = 1-\alpha_\infty >0$. This implies that for any $0<\epsilon<\half$  if
$$\frac{u\tau +u'\tau'}{|us-u's'|}<\epsilon,$$
then, for $u\tau$ and $u\tau'$ both sufficiently large,
\begin{equation}
\label{eq:rboundOrder}
|r_{u,u'}(s,\tau,s',\tau')|\leq (1-2\epsilon)^{2(\alpha_\infty-1)}\left|\frac{\sqrt{u\tau u'\tau'}}{us-u's'}\right|^{2\lambda}.
\end{equation}

Next we focus on the case when $u\sim u'$, $|s-s'|\leq M$ and $|\tau-\tau_0|, |\tau'-\tau^*|\leq \delta(u,u')$ with $\tau^*$ defined in \eqref{tau} and $\lim_{u,u'\to\infty}\delta(u,u')=0$. \rd{ In light of {\bf AI} and {\bf AII}, noting that $\sigma^2$ is bounded over any compact interval, using uniform convergence theorem in \cite{Bingham87} we have that, for $u\sim u'$,
\begin{align*}
&\lim_{u,u'\to\infty}\sup_{|s-s'|\leq M, |\tau-\tau^*|, |\tau'-\tau^*|\leq \delta(u,u')}\left|\frac{\sigma^2(|us-u's'+u\tau|)+\sigma^2(|us - u's'-u'\tau'|)}{\sigma^2(u)}\right.\\
&\quad \quad\quad \quad\quad \quad \left.-|s-s'+\tau^*|^{2\alpha_\infty}-|s-s'-\tau^*|^{2\alpha_\infty}\right|=0,\\
&\lim_{u,u'\to\infty}\sup_{|s-s'|\leq M, |\tau-\tau^*|, |\tau'-\tau^*|\leq \delta(u,u')}\left|\frac{\sigma^2(|us-u's' + u\tau -u'\tau'|)+\sigma^2(|us-u's'|)}{\sigma^2(u)}-2|s-s'|^{2\alpha_\infty}\right|=0,\\
&\lim_{u,u'\to\infty}\sup_{|s-s'|\leq M, |\tau-\tau^*|, |\tau'-\tau^*|\leq \delta(u,u')}\left|\frac{\sigma(u\tau)\sigma(u\tau')}{\sigma^2(u)}-|\tau^*|^{2\alpha_\infty}\right|=0.
\end{align*}}
Hence  for $u\sim u'$
\begin{align}\label{ruu}
\lim_{u,u'\to\infty}\sup_{|s-s'|\leq M, |\tau-\tau^*|, |\tau'-\tau^*|\leq \delta(u,u')}\left|r_{u,u'}(s,\tau,s',\tau')-
g(s-s')\right|=0,
\end{align}
with
$$g(t)=\frac{|t+\tau^*|^{2\alpha_\infty}+|t-\tau^*|^{2\alpha_\infty}-2|t|^{2\alpha_\infty}}{2(\tau^*)^{2\alpha_\infty}}.$$
Note that $g(0)=1$ and for any $0<\delta<1$, there exists $0<c_\delta<1/2$ such that
\begin{align}\label{g(t)}\inf_{|t|<c_{\delta}}g(t)>\delta, \quad \sup_{|t|>\delta}g(t)<1-c_\delta.
\end{align}
The proof of (\ref{g(t)}) is postponed to Appendix.
Following (\ref{ruu}) and (\ref{g(t)}) , we have that  with $u\sim u'$, for $u$ sufficiently large,
\begin{align}
\label{eq:r1}
\inf_{|s-s'|<c_{\delta}, |\tau-\tau^*|, |\tau'-\tau^*|\leq \delta(u,u') }r_{u,u'}(s,\tau, s',\tau')&>\delta/2,\\
\label{eq:r2}
\sup_{|s-s'|>\delta, |\tau-\tau^*|, |\tau'-\tau^*|\leq \delta(u,u')}r_{u,u'}(s,\tau, s',\tau')&<1-c_\delta/2<1.
\end{align}
\subsection{Asymptotics}
Let $\tau^*(u)=(\log m(u))/m(u)$ and $J(u)=\{\tau: |\tau-\tau(u)|\le \tau^*(u)\}$.
Due to the following lemma, while analyzing tail asymptotics of the supremum of $Z_u$, we can restrict the considered domain of $(s,\tau)$ to a strip $J(u)$.
\begin{lem}[\cite{Debicki16}, Lemma 5.6 and Theorem 3.3]
\label{lem:asymptotics}
There exists a positive constant $C$ such that for any $v,T > 0$,
\begin{equation}
\label{eq:tauBigger}
\prob{ \sup_{(s,\tau)\in[0,T]\times (J(u))^c} Z_u(s,\tau)>m(u)}
\le
C T
\frac{u^{\gamma}}{m(u)}\Psi(m(u))
\exp\left(- \frac{b}{4} \log^2 (m(u))\right),
\end{equation}
where $b=B/(2A)$. Furthermore,
for any $T>0$ such that, there exist $c\in(0,\half)$ and $H'\in(-\gamma/2,0)$, such that $u^{H'} <T< \exp(c m^2(u))$ for $u$ sufficiently large,
\begin{equation}
\label{eq:tauSmaller}
\prob{ \sup_{(s,\tau)\in[0,T]\times J(u)} Z_u(s,\tau)>m(u)}
=
(\mathcal{H}_{\eta_{\alpha_\infty}})^2\sqrt{\frac{2A\pi}{B}}\zeta_{\alpha_\infty} T\frac{u^{\gamma}}{m(u)}
\Psi(m(u))(1+o(1)),
\end{equation}
where
\[
\eta_{\alpha_\infty}(t)=
\left\{\begin{array}{cc}
	B_{\alpha_\infty}(t) & \alpha_\infty>1/2\\
	\frac{1+c\tau^*}{\sqrt{2}A_\infty \tau^*}
	X(\overleftarrow{\sigma}(\frac{\sqrt{2}A_\infty \tau^*}{1+c\tau^*}) t) & \alpha_\infty=1/2\\
	B_{\alpha_0}(t)& \alpha_\infty<1/2
\end{array}\right.,
\quad
\zeta_{\alpha_\infty}=
\left\{\begin{array}{cc}
\left(\frac{\sqrt{2A_\infty}(\tau^*)^{2\alpha_\infty}}{1+c\tau^*}\right)^{-2/\alpha_\infty}& \alpha_\infty>1/2\\
\left(\overleftarrow{\sigma}\left(\frac{\sqrt{2}A_\infty \tau^*}{1+c\tau^*}\right)\right)^{-2}& \alpha_\infty=1/2\\
\left(\frac{\sqrt{2}A_\infty (\tau^*)^{2\alpha_\infty}}{\sqrt{A_0}(1+c\tau^*)}\right)^{-2/\alpha_0}& \alpha_\infty<1/2
\end{array}\right.,
\]
with $\gamma$ defined in \eqref{eq:fp&gamma} and $\tau^*$ given by \eqref{tau}.
\end{lem}

\subsection{Discretization}
\label{sec:disc}
For a fixed $T,\theta>0$ and some $u>0$, let us define a discretization of the set $[0,T]\times J(u)$ as follows
\begin{align*}
s_{l} &= lq(u),\quad 0\le l\le L,\quad L = [T/q(u)],\quad q(u) = \theta \frac{\Delta(u)}{u}, \quad \Delta(u)=\overleftarrow{\sigma}\left(\frac{\sqrt{2}\sigma^2(u\tau^*)}{u(1+c\tau^*)}\right)\\
\tau_n &= \tau(u)+ nq(u),\quad 0\le |n|\le N,\quad N = [\tau^*(u)/q(u)], \quad E_{l,n}(u)=[s_l,s_{l+1}]\times[\tau_n, \tau_{n+1}].
\end{align*}
Along the similar lines as in \citep[Lemma 6]{Husler04b} we get the following lemma.
\begin{lem}
\label{lem:discreate_approx}
There exist positive constants $K_1,K_2,u_0>0$, such that, for any $\theta=\theta(u)>0$ with $\lim_{u\to\infty}\theta(u)=0$, $u\ge u_0$ and  $\eta\in (0, \min(\alpha_0,\alpha_\infty))$
\begin{align*}
\mathbb{P}\left(\max_{\substack{0\le l\le L\\0\le |n|\le N}}Z_u(s_l,\tau_n)\le
m(u) - \frac{\theta^{\eta}}{m(u)},
\sup_{\substack{s\in[0,T]\\\tau\in J(u)}} Z_u(s,\tau)>m(u)\right)
\le
K_1\frac{u^{\gamma}}{m(u)}\Psi(m(u))e^{-\frac{\theta^{-2H}}{K_2}}
\end{align*}
with $H\in (0, \min(\alpha_0,\alpha_\infty)-\eta)$.
\end{lem}
\begin{proof}
Conditioning on $Z_u(s_l,\tau_n)=m(u)-\frac{\theta^\eta}{m(u)}$, we have for $u$ sufficiently large
\begin{align*}
 \mathbb{P}& \left(Z_u(s_l,\tau_n)\leq m(u)-\frac{\theta^\eta}{m(u)}, \sup_{(s,\tau)\in E_{l,n}(u)}Z_u(s,\tau)>m(u)\right)\\
& =\int_{\theta^\eta}^\infty\frac{1}{\sqrt{2\pi}m(u)\sigma_u(\tau_n)}e^{-\frac{(m(u)-y/m(u))^2}{2\sigma_u^2(\tau_n)}}
\mathbb{P}\left(\sup_{(s,\tau)\in E_{l,n}(u)}Z_u(s,\tau)>m(u) \Bigl\lvert Z_u(s_l,\tau_n)=m(u)-\frac{y}{m(u)}\right)\D y\\
& \leq \frac{K}{\sqrt{2\pi}m(u)}e^{-\frac{m^2(u)}{2\sigma_u^2(\tau_n)}}\int_{\theta^\eta}^\infty e^{2y}
\mathbb{P}\left(\sup_{(s,\tau)\in E_{l,n}(u)}Z_u(s,\tau)-m(u)>0 \Bigl\lvert Z_u(s_l,\tau_n)=m(u)-\frac{y}{m(u)}\right) \D y.
\end{align*}
Moreover,
\begin{align*}
Z_u(s,\tau)-m(u)\Bigl\lvert Z_u(s_l,\tau_n)=m(u)-\frac{y}{m(u)}
\stackrel{d}{=}Y_u(s,\tau)+h(u,y),
\end{align*}
holds for $(s,\tau)\in E_{l,n}(u)$, where
\begin{align*}
Y_u(s,\tau)&=Z_u(s,\tau)-r_u(s,\tau,s_l,\tau_n)\frac{\sigma_u(\tau)}{\sigma_u(\tau_n)}Z_u(s_l,\tau_n),\\
h(u,y)&=r_u(s,\tau,s_l,\tau_n)\frac{\sigma_u(\tau)}{\sigma_u(\tau_n)}\left(m(u)-\frac{y}{m(u)}\right)-m(u).
\end{align*}
Taylor's formula gives that
\begin{align*}
m(u)h(u,y)
&=-m^2(u)(1-r_u(s,\tau,s_l,\tau_n))
-m^2(u)r_u(s,\tau,s_l,\tau_n)\left(1-\frac{\sigma_u(\tau)}{\sigma_u(\tau_n)}\right)-r_u(s,\tau,s_l,\tau_n)\frac{\sigma_u(\tau)}{\sigma_u(\tau_n)}y\\
&\leq-m^2(u)r_u(s,\tau,s_l,\tau_n)\frac{\dot{\sigma}_u(\tau)(\tau_n-\tau)+(1/2)\ddot{\sigma}_u(v)(\tau-\tau_n)^2}{\sigma_u(\tau_n)}-
r_u(s,\tau,s_l,\tau_n)\frac{\sigma_u(\tau)}{\sigma_u(\tau_n)}y,
\end{align*}
with $v\in (\tau_n, \tau)$.
Using the fact that $\dot{\sigma}_u(\tau(u))=0$ and $\sup_{\tau\in J(u)}|\ddot{\sigma}_u(\tau)|\leq \frac{2B}{A} $ for $u$ sufficiently large, by Taylor's formula, we have
\begin{align*}
m^2(u)|\dot{\sigma}_u(\tau)(\tau_n-\tau)|&= m^2(u)|(\dot{\sigma}_u(\tau)-\dot{\sigma}_u(\tau(u)))(\tau_n-\tau)|\\
&= m^2(u)|\tau_n-\tau||\ddot{\sigma}(v_1)(\tau(u)-\tau)|\\
&\leq\frac{2B}{A}m^2(u)q(u)\tau^*(u)\\
&=\theta \frac{2B}{A}\frac{m(u)\Delta(u)}{u}\log m(u),
\end{align*}
with $v_1\in (\tau, \tau(u))$.
Note that by {\bf AI-AII}
$$\frac{m(u)\Delta(u)}{u}\log m(u)\sim \mathbb{Q}u^{v_2}\log u, \quad \text{ with} \quad v_2=\left\{\begin{array}{cc}
2-\alpha_\infty-1/\alpha_\infty& \alpha_\infty\geq 1/2\\
\frac{2\alpha_\infty-1}{\alpha_0}-\alpha_\infty& \alpha_\infty<1/2
\end{array}\right.. $$
Since $v_2<0$ for all $\alpha_\infty\in (0,1]$, then
\begin{align*}
m^2(u)|\dot{\sigma}_u(\tau)(\tau_n-\tau)|=o(\theta), \quad u\to\infty.
\end{align*}
\begin{align*}
m^2(u)|\ddot{\sigma}_u(\theta)(\tau-\tau_n)^2|\leq K\left(\frac{m(u)\Delta(u)}{u}\right)^2 \theta^2=o(\theta^2).
\end{align*}
Due to the fact that $y\geq \theta^\eta$ with $0<\eta<1$, we have
\begin{align*}
h(u,y)\leq -y(1+o(1)).
\end{align*}
Consequently, for $u$ sufficiently large
\begin{align*}
\mathbb{P}&\left(\sup_{(s,\tau)\in E_{l,n}(u)}Z_u(s,\tau)-m(u)>0 \Bigl\lvert Z_u(s_l,\tau_n)=m(u)-\frac{y}{m(u)}\right)\\
&\leq \mathbb{P}\left(\sup_{(s,\tau)\in [0,1]^2}m(u)\frac{Y_u(s_l+qs,\tau_n+q\tau)}{\sigma_u(\tau_n+q\tau)}>\frac{y}{\sup_{\tau\in [0,1]}\sigma_u(\tau_n+q\tau)}(1+o(1)) \right)\\
&\leq \mathbb{P}\left(\sup_{(s,\tau)\in [0,1]^2}m(u)\frac{Y_u(s_l+qs,\tau_n+q\tau)}{\sigma_u(\tau_n+q\tau)}>\frac{y}{2} \right)
\end{align*}
\rd{
 By (\ref{cor0})  for  $u$ large enough
\begin{align*}
m^2(u)&\Var\left(\frac{Y_u(s_l+qs,\tau_n+q\tau)}{\sigma_u(\tau_n+q\tau)}-\frac{Y_u(s_l+qs',\tau_n+q\tau')}{\sigma_u(\tau_n+q\tau')}\right)\\
&\leq8 m^2(u)(1-r_u(s_l+qs,\tau_n+q\tau,s_l+qs',\tau_n+q\tau'))\\
&\leq 16 m^2(u)\frac{\sigma^2(uq(u)|s-s'|)+\sigma^2(uq(u)|s+\tau-s'-\tau'|)}{2\sigma^2(u\tau^*)}\\
& \leq K\frac{\sigma^2(\Delta(u)\theta|s-s'|)+\sigma^2(\Delta(u)\theta|s+\tau-s'-\tau'|)}{\sigma^2(\Delta(u))}\\
& \leq K\left(\frac{h(\Delta(u)\theta|s-s'|)}{h(\Delta(u))}\theta^{2\eta'}|s-s'|^{2\eta'}+
\frac{h(\Delta(u)\theta|s+\tau-s'-\tau'|)}{h(\Delta(u))}\theta^{2\eta'}|s+\tau-s'-\tau'|^{2\eta'}\right)\\
&\leq K\left(\frac{h(\Delta(u)\theta|s-s'|)}{h(\Delta(u))}+
\frac{h(\Delta(u)\theta|s+\tau-s'-\tau'|)}{h(\Delta(u))}\right)\theta^{2\eta'}(|s-s'|^{2\eta'}+|\tau-\tau'|^{2\eta'}), \quad s,s',\tau, \tau'\in [0,1],
\end{align*}
where   $h(t)=\frac{\sigma^2(t)}{t^{2\eta'}}$ and $\eta'\in (\eta,\min(\alpha_0,\alpha_\infty)).$ Then it follows from {\bf AI} and {\bf AII} that $h(t)>0, t>0$ is a regularly varying function at both $0$ and $\infty$ with indices $2(\alpha_0-\eta')>0$ and $2(\alpha_\infty-\eta')>0$ respectively; see \cite{Bingham87} for the definition and properties of regularly varying functions. Next we focus on the boundedness of  $\sup_{s,s'\in[0,1]}\frac{h(\Delta(u)\theta|s-s'|)}{h(\Delta(u))}$.
If $\lim_{u\to\infty} \Delta(u)=\infty $, noting that  $h$  is bounded over any compact interval, then
 uniform convergence theorem in \cite{Bingham87} gives that
 $$\lim_{u\to\infty}\sup_{s,s'\in[0,1]}\left|\frac{h(\Delta(u)\theta|s-s'|)}{h(\Delta(u))}
 -(\theta|s-s'|)^{2(\alpha_\infty-\eta')}\right|=0,$$
 implying that there exists $K_1>0$ such that for $u$ large enough
 $$\sup_{s,s'\in[0,1]}\frac{h(\Delta(u)\theta|s-s'|)}{h(\Delta(u))}<K_1.$$
 For the case $\lim_{u\to\infty} \Delta(u)=0$, uniform convergence theorem in \cite{Bingham87} can similarly show that the above argument holds. For $\lim_{u\to\infty} \Delta(u)\in (0,\infty)$, it is obvious that
 $$\lim_{u\to\infty}\sup_{s,s'\in[0,1]}h(\Delta(u)\theta|s-s'|)=0, \quad \lim_{u\to\infty}h(\Delta(u))\in (0,\infty).$$
 Thus the boundedness of $\sup_{s,s'\in[0,1]}\frac{h(\Delta(u)\theta|s-s'|)}{h(\Delta(u))}$ also holds.
 The boundedness of $\sup_{s,s',\tau,\tau'\in[0,1]}\frac{h(\Delta(u)\theta|s+\tau-s'-\tau'|)}{h(\Delta(u))}$ can be given similarly.
 Thus we have that
 \begin{align*}
 m^2(u)&\Var\left(\frac{Y_u(s_l+qs,\tau_n+q\tau)}{\sigma_u(\tau_n+q\tau)}-\frac{Y_u(s_l+qs',\tau_n+q\tau')}{\sigma_u(\tau_n+q\tau')}\right)\\
&\leq K\theta^{2\eta'}(|s-s'|^{2\eta'}+|\tau-\tau'|^{2\eta'}), \quad s,s',\tau, \tau'\in [0,1],
\end{align*}
with $\eta'\in (\eta,\min(\alpha_0,\alpha_\infty)).$}
  Similarly
\begin{align*}
\sup_{s,\tau\in [0,1]}m^2(u)\Var\left(\frac{Y_u(s_l+qs,\tau_n+q\tau)}{\sigma_u(\tau_n+q\tau)}\right)\leq K\theta^{2\eta'}.
\end{align*}
Hence
in light of Piterbarg inequality (\citep[Theorem 8.1]{Piterbarg96} or \citep[Lemma 5.1]{Debicki16}), we have for $u$ sufficiently large
\begin{align*}
\mathbb{P}&\left(\sup_{(s,\tau)\in [0,1]^2}m(u)\frac{Y_u(s_l+qs,\tau_n+q\tau)}{\sigma_u(\tau_n+q\tau)}>\frac{y}{2} \right)\\
&\mathbb{P}\left(\sup_{(s,\tau)\in [0,1]^2}\theta^{-\eta'}m(u)\frac{Y_u(s_l+qs,\tau_n+q\tau)}{\sigma_u(\tau_n+q\tau)}>\frac{y}{2}\theta^{-\eta'} \right)\\
& \leq K_1(y\theta^{-\eta'})^{2/\eta'-1}e^{-\frac{(y\theta^{-\eta'})^2}{K}}.
\end{align*}
Consequently,
\begin{align*}
\mathbb{P}&\left(Z_u(s_l,\tau_n)\leq m(u)-\frac{\theta^\eta}{m(u)}, \sup_{(s,\tau)\in E_{l,n}(u)}Z_u(s,\tau)>m(u)\right)\\
& \leq\frac{K_1}{\sqrt{2\pi}m(u)}e^{-\frac{m^2(u)}{2\sigma_u^2(\tau_n)}}\int_{\theta^\eta}^\infty e^{2y}
(y\theta^{-\eta'})^{2/\eta'-1}e^{-\frac{(y\theta^{-\eta'})^2}{K}}dy\\
& \leq\frac{K_1}{\sqrt{2\pi}m(u)}e^{-\frac{m^2(u)}{2\sigma_u^2(\tau_n)}}\theta^{\eta'}\int_{\theta^{\eta-\eta'}}^\infty e^{2y\theta^{\eta'}}
y^{2/\eta'-1}e^{-\frac{y^2}{K}}dy\\
& \leq\frac{K_1}{\sqrt{2\pi}m(u)}e^{-\frac{m^2(u)}{2\sigma_u^2(\tau_n)}}
e^{-\frac{\theta^{2(\eta-\eta')}}{K_2}}.
\end{align*}
 Using the above inequality and (\ref{variance}), we have  that
\begin{align*}
\mathbb{P}&\left(\max_{\substack{0\le l\le L\\0\le |n|\le N}}Z_u(s_l,\tau_n)\le
m(u) - \frac{\theta^\eta}{m(u)},
\sup_{(s,\tau)\in[0,T]\times J(u)} Z_u(s,\tau)>m(u)\right)\\
& \leq \sum_{0\leq l\leq L,|n|\leq N}\mathbb{P}\left(Z_u(s_l,\tau_n)\leq m(u)-\frac{\theta^\eta}{m(u)}, \sup_{(s,\tau)\in E_{l,n}(u)}Z_u(s,\tau)>m(u)\right)\\
& \leq \sum_{0\leq l\leq L,|n|\leq N}\frac{K_1}{\sqrt{2\pi}m(u)}e^{-\frac{m^2(u)}{2\sigma_u^2(\tau_n)}}
e^{-\frac{\theta^{2(\eta-\eta')}}{K_2}}\\
& \leq L\frac{K_1}{\sqrt{2\pi}m(u)}e^{-\frac{\theta^{2(\eta-\eta')}}{K_2}}\sum_{|n|\leq N}e^{-\frac{m^2(u)(1+B(nq)^2/(4A))}{2}}\\
&\leq K_1\left(\frac{u}{m(u)\Delta(u)}\right)^2e^{-\frac{m^2(u)}{2}}\theta^{-2}e^{-\frac{\theta^{2(\eta-\eta')}}{K_2}}\\
&\leq K_1\frac{u^{\gamma}}{m(u)}\Psi(m(u))e^{-\frac{\theta^{2(\eta-\eta')}}{K_2}}.
\end{align*}
This completes the proof.
\end{proof}

\rd{Finally, by following the same arguments as in \cite[Theorems 3.3]{Debicki16} with the supremum functional substituted
by its discrete counterpart, the maximum, we state the following result. Note that the asymptotic result below is a discrete version of \eqref{eq:tauSmaller} in \autoref{lem:asymptotics}.}
\begin{lem}
\label{lem:disc_asymp}
For any $T, \theta>0$, as $u\toi$,
\begin{align*}
\prob{ \max_{\substack{0\le l\le L\\0\le |n|\le N}} Z_u(s_l,\tau_n) > m(u)}
= (\mathcal{H}_{\eta_{\alpha_\infty}}^{\theta})^2\sqrt{\frac{2A\pi}{B}}\zeta_{\alpha_\infty} T\frac{u^{\gamma}}{m(u)}\Psi(m(u))(1+o(1)),
\end{align*}
where $\mathcal H_{\eta_{\alpha_\infty}}^\theta=\lim_{S\to\infty}
S^{-1}\ee \exp\left(\sup_{t\in\theta \mathbb Z \cap [0,S]}\left(\sqrt 2 \eta_{\alpha_\infty}(t) - \Var(\eta_{\alpha_\infty}(t))\right)\right)$.
\end{lem}
\rd{By the monotone convergence theorem, it follows that $\mathcal H_{\eta_{\alpha_\infty}}^\theta\to \mathcal H_{\eta_{\alpha\infty}}$ as $\theta \to 0$, since $H_{\eta_{\alpha\infty}}$ is a positive, finite constant and $\eta_{\alpha_\infty}$ has almost surely continuous sample paths.
Consequently, when the discretization parameter $\theta$ decreases to zero so that the number of discretization points grows to infinity, we recover \eqref{eq:tauSmaller}.}

\section{Auxiliary Lemmas}
\label{sec:auxLem}
We begin with some auxiliary lemmas that are later needed in the proofs.
The first lemma is \citep[Theorem 4.2.1]{Leadbetter83}.
\begin{lem}[Berman's inequality]
Suppose $\xi_1,\ldots,\xi_n$ are standard normal variables with covariance matrix $\Lambda^1=(\Lambda^1_{i,j})$ and $\eta_1,\ldots,\eta_n$ similarly with covariance matrix $\Lambda^0=(\Lambda_{i,j}^0)$. Let $\rho_{i,j}=\max(|\Lambda_{i,j}^1|,|\Lambda_{i,j}^0|)$ and let $u_1,\ldots,u_n$ be real numbers. Then,
\begin{align*}
\mathbb P\left(
\bigcap_{j=1}^n
\{\xi_j\le u_j\}\right) &- \prob{\bigcap_{j=1}^n\{\eta_j\le u_j\}}
\\
&\le
\frac{1}{2\pi}\sum_{1\le i<j\le n}\left(\Lambda_{i,j}^1-\Lambda_{i,j}^0\right)^+
(1-\rho_{i,j}^2)^{-\half}\exp\left(-\frac{u_i^2+u_j^2}{2(1+\rho_{i,j})}\right).
\end{align*}
\end{lem}
The following lemma is a general form of the Borel-Cantelli lemma; cf. \citep{Spitzer64}.
\begin{lem}[Borel-Cantelli lemma]
Consider a sequence of event $\{E_k\}_{k=0}^\infty$. If
\[
\sum_{k=0}^\infty \prob{E_k} < \infty,
\]
then $\prob{E_n\text{ i.o.}} = 0$. Whereas, if
\[
\sum_{k=0}^\infty \prob{E_k} = \infty\quad\text{and}\quad
\liminf_{n\toi}\frac{\sum_{1\le k\ne t\le n}\prob{E_k E_t}}{\left(\sum_{k=1}^n\prob{E_k}\right)^2}\le 1,
\]
then $\prob{E_n\text{ i.o.}} = 1$.
\end{lem}

\begin{lem}
\label{lem:bound}
For any $\varepsilon\in(0,1)$, there exist positive constants $K$ and $\rho$ depending only on $\varepsilon, \alpha_0,\alpha_\infty$ and $p$ such that
\[
\prob{\sup_{S< t\le T} \frac{Q_X(t)}{f_p(t)}\le 1}
\le
\exp
\left(
-\rd{(1-\varepsilon)}
\int_{S+f_p(S)}^{T}
\frac{1}{f_p(u)}
\prob{\sup_{t\in [0,f_p(u)]} Q_X(t) >  f_p(u)}\D u
\right) + K S^{-\rho},
\]
for any $T-f_p(S)\ge S\ge K$, with $f_p(T)/f_p(S)\le \mathcal C$ and $\mathcal C$ being some universal positive constant.
\end{lem}
\begin{proof}
Let $\varepsilon\in(0,1)$ be some positive constant. For the remainder of the proof let $K$ and $\rho$ be two positive constants depending only on $\varepsilon, \alpha_0,\alpha_\infty$ and $p$ that may differ from line to line. For any $k\ge 0$ put $s_0 = S$, $y_0 = f_p(s_0)$, $t_0 = s_0 + y_0$, $x_0 = f_p(t_0)$ and
\begin{align*}
s_k & = t_{k-1} + \varepsilon x_{k-1},
\quad y_k = f_p(s_k),
\quad t_k = s_k + y_k,
\quad x_k = f_p(t_k),\\
\quad I_k &=(s_k,t_k],
\quad \tilde I_k = \frac{I_k}{x_k} = (\tilde s_k, \tilde t_k],
\quad |\tilde I_k| = \frac{y_k}{x_k}.
\end{align*}
From this construction, it is easy to see that the intervals $I_k$ are disjoint. Furthermore, $\delta(I_k,I_{k+1})=\varepsilon x_k$, and $1-\varepsilon\le y_k/x_k\le 1$, for any $k\ge0$ and sufficiently large $S$.
Note that, for any $k\ge 0$, $|I_k|\ge f_p(S)$, therefore if $T(S,\varepsilon)$ is the smallest number of intervals  $\{I_k\}$ needed to cover $[S,T]$, then $T(S,\varepsilon)\le[(T-S)/(f_p(S)(1+\varepsilon))]$. Moreover, since $f_p(T)/f_p(S)$ is bounded by the constant $\mathcal C>0$ not depending on $S$ and $\varepsilon$, it follows that, $x_k/x_t\le \mathcal C$ for any $0\le t<k\le T(S,\varepsilon)$.

Now let us introduce a discretization of the set $\tilde I_k\times J(x_k)$ as in \autoref{sec:disc}.
That is, for some $\theta=\frac{\Delta(S)}{S}$, define grid points
\begin{align*}
\label{eq:disc100}
s_{k,l} &= \tilde s_k + lq_k,\quad 0\le l\le L_k,\quad  L_k = [(1-\varepsilon)/ q_k],
\quad q_k  = \theta \frac{\Delta(x_k)}{x_k},\\
\tau_{k,n} &= \tau(x_k)+ nq_k, \quad 0\le |n| \le N_k,\quad N_k = [\tau^*(x_k)/q_k].
\end{align*}

Since $f_p$ is an increasing function, it easily follows that,
\begin{align*}
\mathbb P  &\Bigg(\sup_{S< t\le T}\frac{Q_{X}(t)}{f_p(t)}\le 1\Bigg)
\le
\prob{\bigcap_{k=0}^{T(S,\varepsilon)}\left\{\sup_{t \in I_k} Q_{X}(t) \le x_k\right\}}
\le
\prob{\bigcap_{k=0}^{T(S,\varepsilon)}\left\{\sup_{\substack{s \in I_k/x_k \\ \tau\in J(x_k)}} Z_{x_k}(s,\tau) \le m(x_k)\right\}}\\
&\le
\prob{\bigcap_{k=0}^{T(S,\varepsilon)}\left\{\max_{\substack{ 0\le l\le L_k\\\quad 0\le |n| \le N_k}} Z_{x_k}(s_{k,l},\tau_{k,n})\le m(x_k)\right\}}\\
&\le
\prod_{k=0}^{T(S,\varepsilon)}\prob{\max_{\substack{ 0\le l\le L_k\\\quad 0\le |n| \le N_k}} Z_{x_k}(s_{k,l},\tau_{k,n})\le m(x_k)}
+\sum_{0\le t<k\le T(S,\varepsilon)}C_{k,t} =: P_1 + P_2,
\end{align*}
where the last inequality follows from Berman's inequality with
\[
C_{k,t}=
 \sum_{\substack{0\le l\le L_k\\0\le p\le L_t}}\sum_{\substack{|n|\le N_k\\|m|\le N_t}}
\frac{|r_{x_k,x_t}(s_{k,l},\tau_{k,n},s_{t,p},\tau_{t,m})|}{\sqrt{1- r^2_{x_k,x_t}(s_{k,l},\tau_{k,n},s_{t,p},\tau_{t,m})}}
\exp\left(-\frac{\half(m^2(x_k)+m^2(x_t))}{1+|r_{x_k,x_t}(s_{k,l},\tau_{k,n},s_{t,p},\tau_{t,m})|}\right).
\]

\noindent\textit{Estimation of $P_1$:}

\vb

Since for any $u$ the process $Z_u$ is stationary in the first variable, from \autoref{lem:disc_asymp}
we have that, as $S\to\infty$ (noting that $\theta=\frac{\Delta(S)}{S}\to 0$)
\begin{align*}
 \prob{\max_{\substack{ 0\le l\le L_k\\\quad 0\le |n| \le N_k}} Z_{x_k}(s_{k,l},\tau_{k,n}) > m(x_k)}\sim \prob{\sup_{(s,\tau)\in \tilde I_k\times J(x_k)} Z_{x_k}(s,\tau) > m(x_k)}
\end{align*}
uniformly with respect to $0\leq k\leq T(S,\varepsilon)$.
Hence  for any $\varepsilon\in(0,1)$, sufficiently large $S$ and small $\theta$,
\begin{align*}
P_1 &\le
\exp\left(
-\sum_{k=0}^{T(S,\varepsilon)} \prob{\max_{\substack{ 0\le l\le L_k\\\quad 0\le |n| \le N_k}} Z_{x_k}(s_{k,l},\tau_{k,n}) > m(x_k)}
\right)\\
&\le
\exp\left(
-\rd{(1-\frac{\varepsilon}{8})}\sum_{k=0}^{T(S,\varepsilon)} \prob{\sup_{(s,\tau)\in \tilde I_k\times J(x_k)} Z_{x_k}(s,\tau) > m(x_k)}
\right)
\end{align*}
Then, by \eqref{eq:basic_transformation} combined with \autoref{lem:asymptotics},
\begin{align*}
P_1
&\le
\exp\left(
-\rd{(1-\frac{\varepsilon}{4})}\sum_{k=0}^{T(S,\varepsilon)} \prob{\sup_{\substack{s\in \tilde I_k\\\tau\ge 0}} Z_{x_k} (s,\tau) > m(x_k)}
\right)\\
&=
\exp\left(
-\rd{(1-\frac{\varepsilon}{4})}\sum_{k=0}^{T(S,\varepsilon)} \prob{\sup_{t\in [0,\frac{y_k}{x_k}f_p(t_k)]} Q_{X}(t) >  f_p(t_k)}
\right)\\
&\le
\exp\left(
-\rd{(1-\frac{\varepsilon}{2})}\sum_{k=0}^{T(S,\varepsilon)} \prob{\sup_{t\in [0,f_p(t_k)]} Q_{X}(t) >  f_p(t_k)}
\frac{f_p(s_k)}{f_p(t_k)}
\right)\\
&\le
\exp\left(
-\rd{(1-\varepsilon)}\int_{S+f_p(S)}^{T}
\frac{1}{f_p(u)}\prob{\sup_{t\in [0,f_p(u)]} Q_{X}(t) >  f_p(u)}\D u
\right).
\end{align*}

\noindent\textit{Estimation of $P_2$:}

\vb

For any $0\le t<k\le T(S,\varepsilon)$, $0\le l\le L_k$, $0\le p\le L_t$, we have
\begin{align*}
x_ks_{k,l}- x_ts_{t,p}
&=
(s_k+ x_k l q_k) - (s_t + x_t p  q_t)\\
&=
\sum_{i=t}^{k-1}(y_i+\varepsilon x_i) + x_k l q_k - x_t p  q_t
\ge
\sum_{i=t}^{k-1}(y_i+\varepsilon x_i) - y_t\\
&\ge \sum_{i=t+1}^{k-1}y_i.
\end{align*}
Recall that $\lambda = 1-\alpha_\infty$. Hence we can find $s_0>2$ such that for $S$ sufficiently large, $k-t\geq 2s_0$, $0\le l\le L_k$, $0\le p\le L_t$, $|n|\leq N_k$ and $|m|\leq N_t$
\begin{align*}
\frac{x_k\tau_{k,n}+x_t\tau_{t,m}}{|x_ks_{k,l}- x_ts_{t,p}|}\leq \frac{x_k(\tau_{k,n}+\tau_{t,m})}{\sum_{i=t+1}^{k-1}y_i}<1/3,
\end{align*}
which applied to \eqref{eq:rboundOrder} indicates that for $k-t\geq s_0$ and $S$ sufficiently large,
\begin{align}\label{upperboundcor}
r^*_{k,t}:&=\sup_{
	\substack{		
		0\le l\le L_k,
		0\le p\le L_t\\
		|n|\le N_k,
		|m|\le N_t
	}
}
|r_{x_k,x_t}(s_{k,l},\tau_{k,n},s_{t,p},\tau_{t,m})|\nonumber\\
&\le 3^{2(1-\alpha_\infty)}
\sup_{
	\substack{
		0\le l\le L_k,
		0\le p\le L_t\\
		|n|\le N_k,
		|m|\le N_t
	}
}
\left|\frac{\sqrt{x_k\tau_{k,n}x_t\tau_{t,m}}}{\sum_{i=t+1}^{k-1}y_i}\right|^{2\lambda}\nonumber\\
&\le 3^{2(1-\alpha_\infty)}
\sup_{
	\substack{
		0\le l\le L_k,
		0\le p\le L_t\\
		|n|\le N_k,
		|m|\le N_t
	}
}
\left|\frac{x_k\tau_{k,n}}{\sum_{i=t+1}^{k-1}y_i}\right|^{\lambda} \left|\frac{x_t\tau_{t,m}}{\sum_{i=t+1}^{k-1}y_i}\right|^{\lambda}\nonumber\\
&\le K \left|\frac{x_t}{\sum_{i=t+1}^{k-1}y_i}\right|^{\lambda}\le K (k-t)^{-\lambda}\le  \frac{\lambda}{4}.
\end{align}
For $1\leq k-t\leq s_0$, it follows that $x_k\sim x_t$,  $\tau_{k,l}\to \tau^*$ and $\tau_{t,p}\to \tau^*$ as $S\to\infty$,  and $s_{k,l}-s_{t,p}\geq \epsilon x_t/x_k>\epsilon/2$ for $S$ sufficiently large.
Therefore, by (\ref{eq:r2}) there exists a positive constant $\zeta\in(0,1)$ depending only on $\varepsilon$ such that for $S$ sufficiently large
\begin{align}\label{upperboundcor1}
\sup_{1\leq k-t\leq s_0}r^*_{k,t}=\sup_{1\leq k-t\leq s_0}\sup_{
	\substack{		
		0\le l\le L_k,
		0\le p\le L_t\\
		|n|\le N_k,
		|m|\le N_t
	}
}|r_{x_k,x_t}(s_{k,l},\tau_{k,n},s_{t,p},\tau_{t,m})|
\le\zeta<1.
\end{align}
Finally, note that; c.f., \eqref{eq:fp&gamma},
\begin{align*}
N_k&\le L_k \le \frac{2(1-\varepsilon)x_k}{\theta\Delta(x_k)}\le K x_k^{2\gamma}
\le
K (\log t_k)^{\frac{\gamma}{(1-\alpha_\infty)}}
,\\
\exp\left(-\frac{m^2(x_k)}{2}\right) &= \frac{(\log t_k)^{p-\frac{\gamma-1}{2(1-\alpha_\infty)}} }{t_k},
\end{align*}
so that
\begin{align*}
P_2 &\le \frac{4}{\sqrt{1-\zeta^2}}
\sum_{0\le t<k\le T(S,\varepsilon)}
L_kL_tN_kN_t r^*_{k,t}
\exp\left(-\frac{m^2(x_k)+m^2(x_t)}{2(1+r^*_{k,t})}\right)\\
&\le
K\left(
\sum_{\substack{0<k-t\le s_0 \\ 0\le t<k\le T(S,\varepsilon)}} +
\sum_{\substack{k-t > s_0 \\ 0\le t<k\le T(S,\varepsilon)}}
\right) (\cdot)\\
&\le
K\Bigg(
\sum_{k=0}^\infty
x_k^{8\gamma}
\exp\left(-\frac{ m^2(x_k)}{1+\zeta}\right)
+
\sum_{\substack{k-t > s_0 \\ 0\le t<k\le T(S,\varepsilon)}}
(x_kx_t)^{4\gamma}
(k-t)^{-\lambda}
\exp\left(-\frac{m^2(x_k)+m^2(x_t)}{2(1+\frac{\lambda}{4})}\right)
\Bigg)\\
&\le
K\left(
\sum_{k=0}^\infty
t_k^{-\frac{2}{1+ \sqrt{\zeta}}}
+
\sum_{\substack{k-t > s_0 \\ 0\le t<k\le T(S,\varepsilon)}}
t_k^{-\frac{1}{1+\frac{\lambda}{2}}}
 t_t^{-\frac{1}{1+\frac{\lambda}{2}}}
(k-t)^{-\lambda}
\right)\\
&\le
K\left(
\sum_{k=[S]}^\infty
k^{-\frac{2}{1+ \sqrt{\zeta}}}
+
\sum_{\substack{[S]\le t<k\le \infty}}
k^{-\frac{1}{1+\frac{\lambda}{2}}}
t^{-\frac{1}{1+\frac{\lambda}{2}}}
(k-t)^{-\lambda}
\right)\\
&\le
K S^{-\rho},
\end{align*}
where the last inequality follows from basic algebra.
\end{proof}
Let $S>0$ be any fixed number, $a_0 = S$, $y_0 = f_p(a_0)$ and $b_0=a_0+y_0$. For $i>0$, define
\begin{equation}
\label{eq:def_ai2}
a_i = b_{i-1},\quad y_i = f_p(a_i), \quad b_i = a_i + y_i,\quad M_i = (a_i, b_i],
\quad \tilde M_i = \frac{M_i}{y_i} = (\tilde a_i, \tilde b_i].
\end{equation}
From this construction, it is easy to see that the intervals $M_i$ are disjoint, $\cup_{j=0}^{i} M_j=(S,b_{i}]$, and $|\tilde M_i| = 1$.
Now let us introduce a discretization of the set $\tilde M_i\times J(y_i)$ as in \autoref{sec:disc}. That is, for  $\theta=\frac{\Delta(S)}{S}$, define grid points
\begin{align}
\label{eq:disc2}
s_{i,l} &= \tilde a_i+ l q_i,\quad 0\le l\le L_i,\quad  L_i = [1/ q_i],
\quad q_i  = \theta \frac{\Delta(y_i)}{y_i},\\
\nonumber
\tau_{i,n} &= \tau(y_i)+ nq_i, \quad 0\le |n| \le N_i,\quad N_i = [\tau^*(y_i)/q_i].
\end{align}
With the above notation, we have the following lemma.
\begin{lem}
\label{lem:lbound}
For any $\varepsilon\in(0,1)$, there exists positive constants $K$ and $\rho$ depending only on $\varepsilon, \alpha_0,\alpha_\infty$ and $p$ such that, with $\theta_i=(m(y_i))^{-4/\hat\alpha}$, where $\hat\alpha = \min(\alpha_0,\alpha_\infty)$,
\begin{align*}
\mathbb P &\left(
\bigcap_{i=0}^{[(T-S)/f_p(S)]}\left\{
\max_{\substack{0\le l\le L_i\\0\le |n| \le N_i}} Z_{y_i}(s_{i,l},\tau_{i,n}) \le m(y_i) - \frac{\theta_i^{\frac{\hat\alpha}{2}}}{m(y_i)}
	\right\}
	\right)
	\\
&\ge
\frac{1}{4}\exp\left(
-(1+\varepsilon)
\int_{S}^{T}
\frac{1}{f_p(u)}\prob{\sup_{t\in [0,f_p(u)]} Q_{X}(t) >  f_p(u)}
	\D u
\right) - K S^{-\rho},
\end{align*}
for any $T-f_p(S)\ge S\ge K$, with $f_p(T)/f_p(S)\le \mathcal C$ and $\mathcal C$ being some universal positive constant.
\end{lem}
\begin{proof}
Put
\[
\hat m(y_i) =  m(y_i) - \frac{\theta_i^{\frac{\hat\alpha}{2}}}{m(y_i)}, \quad I = [(T-S)/f_p(S)].
\]
Similarly as in the proof of \autoref{lem:bound} we find that Berman's inequality implies
\begin{align*}
\mathbb P &\left(
\bigcap_{i=0}^{I}\left\{
\max_{\substack{0\le l\le L_i\\0\le |n| \le N_i}} Z_{y_i}(s_{i,l},\tau_{i,n}) \le m(y_i) - \frac{\theta_i^{\frac{\hat\alpha}{2}}}{m(y_i)}
	\right\}
	\right)\\
&\ge
\prod_{i=0}^{I} \prob{\max_{\substack{0\le l\le L_i\\0\le |n| \le N_i}} Z_{y_i}(s_{i,l},\tau_{i,n}) \le
m(y_i) - \frac{\theta_i^{\frac{\hat\alpha}{2}}}{m(y_i)}}
 -
\sum_{0\le i<j\le I} D_{i,j}=:P_1'+P_2',
\end{align*}
where
\[
D_{i,j} = \frac{1}{2\pi}\sum_{\substack{0\le l\le L_i\\0\le p\le L_j}}\sum_{\substack{|n|\le N_i\\|m|\le N_j}}
\frac{(\tilde r_{y_i,y_j}(s_{i,l},\tau_{i,n},s_{j,p},\tau_{j,m}))^+}{\sqrt{1-\tilde r^2_{y_i,y_j}(s_{i,l},\tau_{i,n},s_{j,p},\tau_{j,m})}}\
\exp\left(-\frac{\half(\hat m^2(y_i)+\hat m^2(y_j))}{1+|\tilde r_{y_i,y_j}(s_{i,l},\tau_{i,n},s_{j,p},\tau_{j,m})|}\right),
\]
with
\[
\tilde r_{y_i,y_j}(s_{i,l},\tau_{i,n},s_{j,p},\tau_{j,m})
= - r_{y_i,y_j}(s_{i,l},\tau_{i,n},s_{j,p},\tau_{j,m}).
\]

\noindent\textit{Estimation of $P_1'$:}

\vb

By \autoref{lem:asymptotics} the correction term $\theta_i^{\hat\alpha/2}/m(y_i)$ does not change the order of asymptotics of the tail of $Z_{y_i}$.
Furthermore, the tail asymptotics of the supremum on the strip $(s,\tau)\in\tilde M_i\times J(y_i)$ are of the same order if $\tau\ge 0$. Hence, for every $\varepsilon>0$,
\begin{align*}
P_1' &
\ge
\frac{1}{4}
 \exp\left(
-\sum_{i=0}^{I} \prob{\max_{\substack{0\le l\le L_i\\0\le |n| \le N_i}} Z_{y_i}(s_{i,l},\tau_{i,n}) > \hat m(y_i)}
\right)\\
&\ge
\frac{1}{4}
 \exp\left(
-\sum_{i=0}^{I} \prob{\sup_{\substack{s\in\tilde M_i\\ \tau\in J(y_i)}} Z_{y_i}(s,\tau) > m(y_i) - \frac{\theta_i^{\frac{\hat\alpha}{2}}}{m(y_i)}}
\right)
\\
&\ge
\frac{1}{4}
 \exp\left(
-(1+\varepsilon)\sum_{i=0}^{I} \prob{\sup_{\substack{s\in\tilde M_i\\ \tau\ge 0}} Z_{y_i}(s,\tau) > m(y_i)}
\right)
\\
&=
\frac{1}{4}
 \exp\left(
-(1+\varepsilon)\sum_{i=0}^{I} \prob{\sup_{t\in[0,f_p(a_i)]} Q_{X}(t) > f_p(a_i)}
\right)\\
&\ge
\frac{1}{4}\exp\left(
-(1+\varepsilon)
\int_{S}^{T}
\frac{1}{f_p(u)}\prob{\sup_{t\in[0,f_p(u)]} Q_{X}(t) > f_p(u)}\D u
\right),
\end{align*}
provided that $S$ is sufficiently large along the same lines as the estimation of $P_1$ in \autoref{lem:bound}.
\vb

\noindent\textit{Estimation of $P_2'$:}

\vb

Clearly, for $j\ge i+2$, and any $0\le l\le L_i$, $0\le p\le L_j$; c.f. \eqref{eq:def_ai2},
\[
y_j s_{j,p} - y_i s_{i,l} = a_j + y_j p q_j - \left(a_i + y_i l  q_i\right) \ge \sum_{k=i+1}^{j-1} y_k.
\]
Hence there exists $s_0\geq 2$ such that for $j-i\geq s_0$, $0\le l\le L_i$, $0\le p\le L_j$, $|n|\leq N_i$, $|m|\leq N_j$ and $S$ sufficiently large
\begin{align*}
\frac{y_j\tau_{j,m}+y_i\tau_{i,n}}{|y_j s_{j,p} - y_i s_{i,l}|}\leq \frac{y_j(\tau_{j,m}+\tau_{i,n})}{\sum_{k=i+1}^{j-1} y_k}\leq \frac{1}{3}.
\end{align*}
Analogously as the derivation of (\ref{upperboundcor}), by \eqref{eq:rboundOrder}  for $j-i\geq s_0$ and $S$ sufficiently large
\[
r^*_{i,j}:=\sup_{
	\substack{		
		0\le l\le L_i,
		0\le p\le L_j\\
		|n|\le N_i,
		|m|\le N_j
	}
} |\tilde r_{y_i,y_j}(s_{i,l},\tau_{i,n},s_{j,p},\tau_{j,m})|\le K (k-t)^{-\lambda}\le  \frac{\lambda}{4},
\]
where $\lambda = 1-\alpha_\infty$. For $1\leq j-i\leq s_0$, it follows that $y_i\sim y_j$,  $\tau_{i,n}\to \tau^*$ and $\tau_{j,m}\to \tau^*$ as $S\to\infty$,  and $s_{i,l}-s_{j,p}\geq  y_{i+1}/y_j>\frac{1}{2}$  for $2\leq j-i\leq s_0$ and $S$ sufficiently large.
Therefore, by (\ref{eq:r2}) there exists a positive constant $\zeta_1\in(0,1)$ depending only on $\varepsilon$ such that for $S$ sufficiently large
\begin{align}\label{upperboundcor2}
\sup_{2\leq j-i\leq s_0}r^*_{i,j}=\sup_{2\leq j-i\leq s_0}\sup_{
	\substack{		
		0\le l\le L_i,
		0\le p\le L_j\\
		|n|\le N_i,
		|m|\le N_j
	}
} | r_{y_i,y_j}(s_{i,l},\tau_{i,n},s_{j,p},\tau_{j,m})|
\le\zeta_1.
\end{align}

Moreover, by
\eqref{eq:r1} there exist positive constants $\delta\in(0,1)$ and $c_\delta\in(0,\half)$, $M<1$, such that,
for sufficiently large $S$,
\[
\inf_{|y_i-y_j|<c_{\delta}, |\tau-\tau^*|, |\tau'-\tau^*|\leq M }
r_{y_i,y_j}(s,\tau,s',\tau')> \frac{\delta}{2}.
\]
Hence  for sufficiently large $S$ and $0\le l\le L_i$, $0\le p\le L_j$, $|n|\leq N_i$, $|m|\leq N_j$
\begin{align}
\label{eq:est2}
(\tilde r_{y_i,y_j}(s_{i,l},\tau_{i,n},s_{j,p},\tau_{j,m}))^+ = 0, &\quad \text{if}\quad
j=i+1, \quad |s_{i,l}-s_{j,p}| \le c_\delta.
\end{align}
By (\ref{eq:r2}) there exits $\zeta_2\in (0,1)$ such that for $S$ sufficiently large and $0\le l\le L_i$, $0\le p\le L_j$, $|n|\leq N_i$, $|m|\leq N_j$
\begin{align}
\label{eq:est3}
|\tilde r_{y_i,y_j}(s_{i,l},\tau_{i,n},s_{j,p},\tau_{j,m})| \le
\zeta_2,
&\quad\text{if}\quad
j=i+1,\quad |s_{i,l}-s_{j,p}| \ge c_\delta.
\end{align}
Let $\zeta=\max(\zeta_1, \zeta_2)$.
Therefore, by \eqref{upperboundcor1}--\eqref{eq:est3} we obtain
\begin{align*}
P_2' &\le
\sum_{\substack{0\le i\le I-1 \\ 1\leq j-i\leq s_0}}\sum_{\substack{0\le l\le L_i\\0\le p\le L_j}}\sum_{\substack{|n|\le N_i\\|m|\le N_j}}
\frac{1}{\sqrt{1- \zeta}}
\exp\left(-\frac{\half(\hat m^2(y_i)+\hat m^2(y_j))}{1+\zeta}\right)\\
&\quad +
\sum_{\substack{0\le i\le I-2 \\ i+s_0 \le j \le I }}\sum_{\substack{0\le l\le L_i\\0\le p\le L_j}}\sum_{\substack{|n|\le N_i\\|m|\le N_j}}
\frac{r^*_{i,j}}{\sqrt{1- r^*_{i,j}}}
\exp\left(-\frac{\half(\hat m^2(y_i)+\hat m^2(y_j))}{1+r^*_{i,j}}\right).
\end{align*}
Completely similar to the estimation of $P_2$ in the proof of \autoref{lem:bound}, we can arrive that there exist positive constants $K$ and $\rho$ such that, for sufficiently large $S$,
\[
P_2'\le K S^{-\rho}.
\]
\end{proof}

The next lemma is a straightforward modification of \citep[Lemma 3.1 and Lemma 4.1]{Watanabe70}, see also \citep[Lemma 1.4]{Qualls71}.
\begin{lem}
\label{lem:v2}
\rd{It is enough to proof \autoref{thm:equiv} for any nondecreasing function $f$ such that,
\begin{equation}
\label{eq:frestricted}
\overleftarrow{m}\left(\sqrt{\log t}\right) \le f(t) \le  \overleftarrow{m}\left(\sqrt{3\log t}\right),
\end{equation}
for all $t\ge T$, and $T$ large enough.}
\end{lem}

\section{Proof of the main results}
\label{sec:Proofs}
\begin{proof}[\bf{Proof of \autoref{thm:equiv}}]
\noindent Note that the case $\mathscr I_f <\infty$ is straightforward and does not need any
additional knowledge on the process $Q_{X}$ apart from the  property of stationarity.
Indeed, consider the sequence of intervals
$M_i$ as in \autoref{lem:lbound}. Then, for any $\varepsilon>0$ and sufficiently large $T$,
\[
\sum_{k = [T]+1}^\infty
\prob{\sup_{t\in M_{k}} Q_{X}(t) > f(a_{k})}
=
\sum_{k= [T]}^\infty \prob{\sup_{t\in[0, f(b_k)]} Q_{X}(t) > f(b_k)}
\le
\mathscr I_f <\infty,
\]
and the Borel-Cantelli lemma completes this part of the proof since $f$ is an increasing function.

Now let $f$ be an increasing function such that $\mathscr I_f\equiv\infty$. With the same notation as in
\autoref{lem:bound} with $f$ instead of $f_p$, we find that, for any $S,\varepsilon,\theta>0$,
\begin{align*}
\prob{Q_{X}(s) > f(s)\text{ i.o.}} &\ge \prob{\left\{\sup_{t\in I_k} Q_{X}(t) > f(t_k)\right\}\quad\text{i.o.}}\\
&\ge
\prob{ \left\{ \max_{\substack{ 0\le l\le L_k\\\quad 0\le |n| \le N_k}} Z_{x_k}(s_{k,l},\tau_{k,n})> m(x_k)\right\}\text{ i.o.}}.
\end{align*}
Let
\[
E_k = \left\{ \max_{\substack{ 0\le l\le L_k\\\quad 0\le |n| \le N_k}} Z_{x_k}(s_{k,l},\tau_{k,n})\le m(x_k)\right\}.
\]
For sufficiently large $S$ and sufficiently small $\theta$; c.f., estimation of $P_1$, we get
\begin{equation}
\label{eq:sum}
\sum_{k=0}^\infty
\prob{ E_k^c}
\ge
\rd{(1-\varepsilon)}
\int_{S+f(S)}^{\infty}
\frac{1}{f(u)}\prob{\sup_{t\in [0,f(u)]} Q_{X}(t) >  f(u)}\D u
=\infty.
\end{equation}
Note that
\[
1-\prob{E_i^c\quad \text{i.o.}}
=
\lim_{m\toi}\prod_{k=m}^\infty\prob{E_k} +
\lim_{m\toi}\left(\prob{\bigcap_{k=m}^\infty E_k} - \prod_{k=m}^\infty\prob{E_k}\right).
\]
The first limit is zero as a consequence of \eqref{eq:sum}, and the second limit will be zero because of the asymptotic independence of the events $E_k$. Indeed, there exist positive constants $K$ and $\rho$, depending only on $\alpha_0,\alpha_\infty,\varepsilon,\lambda$, such that for any $n>m$,
\[
A_{m,n}=\left|\prob{\bigcap_{k=m}^n E_k} - \prod_{k=m}^n\prob{E_k}\right|\le K (S+m)^{-\rho},
\]
by the same calculations as in the estimate of $P_2$ in \autoref{lem:bound} after realizing that, by \autoref{lem:v2},
we might restrict ourselves to the case when
\eqref{eq:frestricted} holds. Therefore $\prob{E_i^c \text{ i.o.}} =1$, which finishes the proof.
\end{proof}

\noindent \textbf{Proof of \autoref{thm:main}:}

\vb

\noindent \rd{Let $\xi_p\equiv \xi_{f_p}$ for short.}

\vb

\noindent\textit{Step 1.}
Let $p>1$, then, for every $\varepsilon\in(0,\frac{1}{4})$,
\[
\liminf_{t\toi}\frac{\xi_p(t) -t}{h_p(t)}\ge -\rd{(1+2\varepsilon)}\quad\text{a.s.}
\]
\begin{proof}
Let $\{T_k:k\ge 1\}$ be a sequence such that $T_k\toi$, as $k\toi$.
Put $S_k =T_k -\rd{(1+2\varepsilon)} h_p(T_k)$. Since $h_p(t)=O(t\log^{1-p} t \log_2 t)$, then, for $p>1$,
$S_k\sim T_k$, as $k\toi$, and from \autoref{lem:bound} it follows that
\begin{align*}
\mathbb P\Big (&\frac{\xi_p(T_k) - T_k}{h_p(T_k)}\le -(1+2\varepsilon)^2\Big )
=
\prob{\xi_p(T_k)\le S_k}
=
\prob{\sup_{S_k<t\le T_k}\frac{Q_{X}(t)}{f_p(t)}< 1}\\
&\le
\exp\left(
-\rd{(1-\varepsilon)}
\int_{S_k+f_p(S_k)}^{T_k}
\frac{1}{f_p(u)}
\prob{\sup_{t\in [0,f_p(u)]} Q_{X}(t) >  f_p(u)}\D u
\right)
+ 2K T_k^{-\rho}.
\end{align*}
Moreover, as $k\toi$,
\begin{align}
\label{eq:logasympt}
\nonumber
\int_{S_k+f_p(S_k)}^{T_k}&
\frac{1}{f_p(u)}
\prob{\sup_{t\in [0,f_p(u)]} Q_{X}(t) >  f_p(u)}\D u\\
&\sim
\rd{(1+2\varepsilon)} h_p(T_k)
\frac{1}{f_p(T_k)}
\prob{
	\sup_{t\in [0,f_p(T_k)]}
		Q_{X}(t) >  f_p(T_k)}
=
\rd{(1+2\varepsilon)} p \log_2 T_k.
\end{align}
Now take $T_k = \exp(k^{1/p})$. Then,
\[
\sum_{k=1}^\infty
\prob{\xi_p(T_k)\le S_k}
\le
2K\sum_{k=1}^\infty k^{-(1+\varepsilon/2)}<\infty.
\]
Hence by the Borel-Cantelli lemma,
\begin{equation}
\label{eq:step}
\liminf_{k\toi}\frac{\xi_p(T_k)-T_k}{h_p(T_k)}\ge -\rd{(1+2\varepsilon)}\quad\text{a.s.}.
\end{equation}
Since $\xi_p(t)$ is a non-decreasing random function of $t$, for every $T_k\le t\le T_{k+1}$, we have
\begin{align*}
\frac{\xi_p(t)-t}{h_p(t)}\ge
\frac{\xi_p(T_k)-T_k}{h_p(T_k)}- \frac{T_{k+1}-T_k}{h_p(T_k)}.
\end{align*}
For $p >1$ elementary calculus implies
\[
\lim_{k\toi}\frac{T_{k+1}- T_k}{ h_p (T_k)}  =0,
\]
so that
\[
\liminf_{t\toi}\frac{\xi_p(t)-t}{h_p(t)}\ge\liminf_{k\toi}\frac{\xi_p(T_k)-T_k}{h_p(T_k)}\quad\text{a.s.},
\]
which finishes the proof of this step.
\end{proof}

\noindent\textit{Step 2.}
Let $p>1$, then, for every $\varepsilon\in(0,1)$,
\[
\liminf_{t\toi}\frac{\xi_p(t)-t}{h_p(t)}\le -(1-\varepsilon)\quad\text{a.s.}
\]
\begin{proof}
As in the proof of the lower bound, put
\[
T_k =\exp(k^{(1+\varepsilon^2)/p}),\quad S_k =T_k -(1-\varepsilon) h_p(T_k),\quad k\ge 1.
\]
Let
\[
B_k = \{\xi_p(T_k)\le S_k\}=\left\{\sup_{S_k<t\le T_k}\frac{Q_{X}(t)}{f_p(t)}<1\right\}.
\]
It suffices to show $\prob{B_n \text{ i.o.}}=1$, that is
\begin{equation}
\label{eq:ub_goal}
\lim_{m\toi}\prob{\bigcup_{k=m}^\infty B_k} =1.
\end{equation}
Let
\[
a_0^k = S_k, \quad y_0^k = f_p(a_0^k),\quad b_0^k = a_0^k + y_0^k,
\]
\[
a_i^k = b_{i-1}^k,\quad y_i^k = f_p(a_i^k),\quad  b_i^k = a_i^k +y_i^k,\quad M_i^k = (a_i^k, b_i^k],\quad
\tilde M_i^k = \frac{M_i^k}{ y_i^k}=(\tilde a_i^k, \tilde b_i^k].
\]
Define $J_k$ to be the biggest number such that $b_{J_k-1}^k\le T_k$ and $b_{J_k}^k> T_k$. Note that $J_k\le [(T_k-S_k)/f_p(S_k)]$.

Since $f_p$ is an increasing function,
\[
B_k \supset \bigcap_{i=0}^{J_k} \left\{\sup_{t\in M_i^k}\frac{Q_{X}(t)}{f_p(t)}<1\right\}
\supset \bigcap_{i=0}^{J_k} \left\{\sup_{t\in M_i^k}Q_{X}(t)< y_i^k\right\}
= \bigcap_{i=0}^{J_k} \left\{\sup_{\substack{s\in \tilde M_i^k \\ \tau\ge 0}} Z_{y_i^k}(s,\tau)< m(y_i^k)\right\}.
\]
Analogously to \eqref{eq:disc2}, define a discretization of the set $\tilde M_i^k\times J(y_i^k)$ as follows
\begin{align*}
s_{i,l}^k &= \tilde a_i^k+ l q_i^k,\quad
0\le l\le L^k_i,\quad  L^k_i = [1/ q_i^k],
\quad q_i^k  = \theta_i^k \frac{\Delta(y_i^k)}{y_i^k},
\quad \theta_i^k  = \left(m(y_i^k)\right)^{-4/\hat\alpha},
\\
\tau_{i,n}^k &= \tau(y_i^k) + n q_i^k,\quad 0\le |n| \le N^k_i,\quad N^k_i = [\tau^*(y_i^k)/q_i^k].
\end{align*}
Recall that $\hat\alpha=\min(\alpha_0,\alpha_\infty)$ and let
\[
A_k=\bigcap_{i=0}^{J_k}\left\{\max_{
\substack{0\le l\le L^k_i\\0\le |n| \le N^k_i}} Z_{y_i^k}(s_{i,l}^k,\tau_{i,n}^k)\le m(y_i^k)-\frac{(\theta_i^k)^{\frac{\hat\alpha}{2}}}{m(y_i^k)}\right\}.
\]
Observe that
\[
\prob{\bigcup_{k=m}^\infty A_k} \le \prob{\bigcup_{k=m}^\infty B_k}+\sum_{k=m}^\infty\prob{A_k\cap B_k^c}.
\]
Furthermore,
\begin{align}
\sum_{k=m}^\infty\prob{A_k\cap B_k^c} &\le \sum_{k=m}^\infty\sum_{i=0}^{J_k}
\prob{\max_{
\substack{0\le l\le L^k_i\\0\le |n| \le N^k_i}} Z_{y_i^k}(s_{i,l}^k,\tau_{i,n}^k)\le m(y_i^k)-\frac{(\theta_i^k)^{\frac{\hat\alpha}{2}}}{m(y_i^k)}, \sup_{\substack{s\in \tilde M_i^k \\ \tau\ge 0}} Z_{y_i^k}(s,\tau)\ge m(y_i^k)}
\nonumber\\
&\le
\sum_{k=m}^\infty\sum_{i=0}^{J_k}
\prob{
	\max_{
\substack{0\le l\le L^k_i\\0\le |n| \le N^k_i}} Z_{y_i^k}(s_{i,l}^k,\tau_{i,n}^k)\le m(y_i^k)-\frac{(\theta_i^k)^{\frac{\hat\alpha}{2}}}{m(y_i^k)},
    \sup_{
    	\substack{
    		s\in \tilde M_i^k \\
    		\tau\in J(y_i^k)}
    	} Z_{y_i^k}(s,\tau)\ge m(y_i^k)}
\nonumber\\
\label{eq:dummy}
&\quad +
\sum_{k=m}^\infty\sum_{i=0}^{J_k}
\prob{
	\sup_{
		\substack{
			s\in \tilde M_i^k \\
			\tau \notin J(y_i^k)
		}
	}
		Z_{y_i^k}(s,\tau)\ge m(y_i^k)
}.
\end{align}
By \autoref{lem:asymptotics} and \autoref{lem:discreate_approx}, for sufficiently large $m$ and some $K_1, K_2>0$, the first sum is bounded from above by
\begin{align*}
\sum_{k=m}^\infty\sum_{i=0}^{J_k}
&
K_1\frac{(y_i^k)^{\gamma}}{m(y_i^k)}\Psi(m(y_i^k))e^{-(m(y_i^k))^{3}/K_2}\\
&\le
\sum_{k=m}^\infty\sum_{i=0}^{J_k}
K_1\prob{\sup_{(s,\tau)\in[0,1]\times\rr_+}Z_{y_i^k}(s,\tau) > m(y_i^k)} e^{-(m(y_i^k))^{3}/K_2}\\
&\le
\sum_{k=m}^\infty\sum_{i=0}^{J_k}
K_1\prob{\sup_{t\in[0,f_p(a_i^k)]}Q_{X}(t) > f_p(a_i^k)} e^{-(\log a_i^k)^{3/2}/K_2}
\\
&\le
K
\int_m^\infty
\frac{\psi(f_p(x))}{f_p(x)}
e^{-\log^{3/2}(x)/K_2}\D x<\infty.
\end{align*}
Note that by \eqref{eq:tauBigger}, for sufficiently large $m$, the term in \eqref{eq:dummy} is bounded from above by
\[
K\sum_{k=m}^\infty\sum_{i=0}^{J_k}
\frac{(y_i^k)^{\gamma}}{m(y_i^k)}\Psi(m(y_i^k))
\exp\left(- \frac{b}{4} \log^2 m(y_i^k)\right)
\le
K
\int_m^\infty
\frac{\psi(f_p(x))}{f_p(x)}
e^{-\frac{b}{4}(\half\log_2 x))^2}\D x<\infty.
\]
Therefore
\[
\lim_{m\toi} \sum_{k=m}^\infty\prob{A_k\cap B_k^c} = 0
\]
and
\[
\lim_{m\toi}\prob{\bigcup_{k=m}^\infty B_k}\ge\lim_{m\toi} \prob{\bigcup_{k=m}^\infty A_k} .
\]
To finish the proof of \eqref{eq:ub_goal}, we only need to show that
\begin{equation}
\label{eq:generalBC}
\prob{A_n \text{ i.o.}}=1.
\end{equation}

Similarly to \eqref{eq:logasympt}, we have
\[
\int_{S_k}^{T_k}
\frac{1}{f_p(u)}
\prob{\sup_{t\in [0,f_p(u)]} Q_{X}(t) >  f_p(u)}\D u
	\sim
(1-\varepsilon) p \log_2 T_k.
\]
Now from \autoref{lem:lbound} it follows that
\[
\prob{A_k} \ge
\frac{1}{4}\exp\left(
-(1-\varepsilon^2)
p \log_2 T_k
\right) - K S_k^{-\rho}
\ge
\frac{1}{8}k^{-(1-\varepsilon^4)},
\]
for every $k$ sufficiently large. Hence,
\begin{equation}
\label{eq:sumA}
\sum_{k=1}^\infty \prob{A_k} =\infty.
\end{equation}

Applying Berman's inequality, we get for $t<k$
\begin{equation}
\label{eq:AkAt}
\prob{A_kA_t}\le \prob{A_k}\prob{A_t} + Q_{k,t},
\end{equation}
where,
\begin{align*}
Q_{k,t} =  \sum_{\substack{0\le i\le J_k \\ 0\le j \le J_t}}&\sum_{\substack{0\le l\le L_i^k\\0\le p\le L_j^t}}\sum_{\substack{|n|\le N_i^k\\|m|\le N_j^t}}
\frac{|r_{y_i^k,y_j^t}(s_{i,l}^k,\tau_{i,n}^k,s_{j,p}^t,\tau_{j,m}^t)|}{\sqrt{1- r^2_{y_i^k,y_j^t}(s_{i,l}^k,\tau_{i,n}^k,s_{j,p}^t,\tau_{j,m}^t)}}\\
&\times \exp\left(
-\frac{( m(y_i^k)-(m(y_i^k))^{-3})^2+( m(y_j^t)-(m(y_j^t))^{-3})^2}{2(1+ |r_{y_i^k,y_j^t}(s_{i,l}^k,\tau_{i,n}^k,s_{j,p}^t,\tau_{j,m}^t)|)}
\right).
\end{align*}
For any $0\le i\le J_k$,
$0\le j\le J_t$, $0\le l\le L_i^k$, $0\le p\le L_j^t$, and $t<k$,
\[
y_i^k s_{i,l}^k-y_j^t s_{j,p}^t =
a_i^k + y_i^k l  q_i^k- \left(a_j^t+ y_j^t p q_j^t\right)
\ge
S_k - T_t \ge S_k - T_{k-1}
\ge
\half(T_k-T_{k-1}),
\]
where the last inequality holds for $k$ large enough since it is easy to see that
\[
\frac{S_{k+1}-T_k}{T_{k+1}-T_k}\sim 1,\quad\as k.
\]
Thus, sufficiently large $k$ and every $0\le t<k$,
and a generic constant $K>0$, similarly to \eqref{upperboundcor} we have,
\[
\sup_{
	\substack{
		0\le i\le J_k\\	
		0\le j\le J_t\\	
		0\le l\le L_i^k,
		0\le p\le L_j^t\\
		|n|\le N_i^k,
		|m|\le N_j^t
	}
}
|r_{y_i^k,y_j^t}(s_{i,l}^k,\tau_n^k,s_{j,p}^t,\tau_{j,m}^t)|
\le
K (T_k - T_{k-1})^{-\lambda/2}
\le\frac{\min(1,\lambda)}{32}.
\]
Therefore, for some generic constant $K$ not depending on $k$ and $t$ which may vary between lines, for every $t<k$ sufficiently large,
\begin{align*}
Q_{k,t} &\le K \sum_{\substack{0\le i\le J_k \\ 0\le j \le J_t}}
L_i^kL_j^tN_i^kN_j^t (T_k-T_{k-1})^{-\lambda/2}
\exp\left(-\frac{(m(y_i^k))^2+(m(y_j^t))^2}{2(1+\frac{\lambda}{16})}\right)\\
&\le
K (T_k-T_{k-1})^{-\lambda/2}
(L_{J_k}^k L_{J_t}^t)^2
\sum_{\substack{0\le i\le J_k \\ 0\le j \le J_t}}
\left(
	a_i^k
	\log^{\frac{\gamma-1}{2(1-\alpha_\infty)}-p} a_i^k
\right)^{-\frac{1}{1+\frac{\lambda}{16}}}
\left(
	a_j^t
	\log^{\frac{\gamma-1}{2(1-\alpha_\infty)}-p}a_j^t
\right)^{-\frac{1}{1+\frac{\lambda}{16}}}
\\
&\le
K (T_k-T_{k-1})^{-\lambda/2}\left(\log T_k\right)^\upsilon
\left(T_k\right)^{\frac{\frac{\lambda}{8}}{1+\frac{\lambda}{8}}}
\left(T_t\right)^{\frac{\frac{\lambda}{8}}{1+\frac{\lambda}{8}}}
\\
&\le
K T_k^{-\lambda/8} \le K \exp(-\lambda k^{(1+\varepsilon^2)/p}/8),
\end{align*}
with $\upsilon>0$ a fixed constant.
Hence we have,
\begin{equation}
\label{eq:sumC}
\sum_{0\le t<k<\infty} Q_{k,t} <\infty.
\end{equation}
Now \eqref{eq:generalBC} follows from \eqref{eq:sumA}-\eqref{eq:sumC} and the general form of the Borel-Cantelli lemma.
\end{proof}

\vb

\noindent\textit{Step 3.}
If $p\in(0,1]$, then, for every $\varepsilon\in(0,\frac{1}{4})$,
\begin{equation}
\label{eq:step3eq1}
\liminf_{t\toi}\frac{\log\left(\xi_p(t)/t\right)}{h_p(t)/t}\ge -\rd{(1+2\varepsilon)}\quad\text{a.s.}
\end{equation}
and
\begin{equation}
\label{eq:step3eq2}
\liminf_{t\toi}\frac{\log\left(\xi_p(t)/t\right)}{h_p(t)/t}\le -(1-\varepsilon)\quad\text{a.s.},
\end{equation}
\begin{proof}
Put
\[
T_k = \exp(k),\quad S_k = T_k \exp\left(-\rd{(1+2\varepsilon)}h_p(T_k)/T_k\right).
\]
Proceeding the same as in the proof of \eqref{eq:step}, one can obtain that
\[
\liminf_{k\toi}\frac{\log\left(\xi_p(T_k)/T_k\right)}{h_p(T_k)/T_k}\ge -\rd{(1+2\varepsilon)}\quad\text{a.s.}
\]
On the other hand it is clear that
\[
\liminf_{t\toi}\frac{\log\left(\xi_p(t)/t\right)}{h_p(t)/t}
=
\liminf_{k\toi}\frac{\log\left(\xi_p(T_k)/T_k\right)}{h_p(T_k)/T_k}\quad\text{a.s.}
\]
since
\[
\liminf_{k\toi}\frac{\log\left(T_k/T_{k+1}\right)}{h_p(T_k)/T_k}=0.
\]
This proves \eqref{eq:step3eq1}.

Let
\[
T_k = \exp\left(k^{1+\varepsilon^2}\right),\quad S_k = T_k\exp\left(-(1-\varepsilon)h_p(T_k)/T_k\right).
\]
Noting that
\[
\frac{S_{k+1}-T_k}{S_{k+1}}\sim 1\quad\as k,
\]
along the same line as in the proof of \eqref{eq:ub_goal}, we also have
\[
\liminf_{k\toi}\frac{\log\left(\xi_p(T_k)/T_k\right)}{h_p(T_k)/T_k}\le -(1-\varepsilon)\quad\text{a.s.},
\]
which proves \eqref{eq:step3eq2}.
\end{proof}

\section{Appendix}
\rd{{\bf Proof of (\ref{g(t)})}. Let $g_1(t)=g(\tau^*t)$. Then it suffices to prove the claim in (\ref{g(t)}) for
$$g_1(t)=\frac{|1+t|^{2\alpha_\infty}+|1-t|^{2\alpha_\infty}-2|t|^{2\alpha_\infty}}{2}.$$
Note that $g_1(t)=g_1(-t), t\geq 0$, it is sufficient to prove the argument for $t\geq 0$.
We distinguish three scenarios: $0<\alpha_\infty<1/2$, $\alpha_\infty=1/2$ and $1/2<\alpha_\infty<1$.\\
We first focus on $\alpha_\infty=1/2$. If $\alpha_\infty=1/2$, then
\begin{align*}
g_1(t)=\left\{\begin{array}{cc}
1-t& 0\leq t\leq 1\\
0 & t\geq 1,
\end{array}
\right.
\end{align*}
which implies that (\ref{g(t)}) holds for $g_1(t)$.\\
Next we consider $0<\alpha_\infty<1/2$.  For $0<t\leq 1$, the first derivative of $g_1$
$$\dot{g}_1(t)=\alpha_\infty\left((1+t)^{2\alpha_\infty-1}-(1-t)^{2\alpha_\infty-1}-2t^{2\alpha_\infty-1}\right)<0.$$
Moreover, for $t>1$, by the convexity of $t^{2\alpha_\infty-1}$
$$\dot{g}_1(t)=\alpha_\infty\left((1+t)^{2\alpha_\infty-1}
+(t-1)^{2\alpha_\infty-1}-2t^{2\alpha_\infty-1}\right)>0.
$$
Additionally, direct calculation shows that $\lim_{t\to\infty} g_1(t)=0.$
This means that for $0<\alpha_\infty<1/2$, $g_1(t)$ is strictly decreasing over $(0,1)$ and increasing over $(1,\infty)$ with $g_1(0)=1$, $g_1(1)<0$ and $\lim_{t\to\infty} g_1(t)=0.$  This implies that for any $0<\delta<1$,
$$\sup_{t>\delta}g_1(t)<1.$$
Thus (\ref{g(t)})  holds for $g_1$ with $0<\alpha_\infty<1/2$.\\
Finally, we focus on $1/2<\alpha_\infty<1$.
For $0<t<1$, using the fact that $s^{2\alpha_\infty-2}$ is strictly decreasing over $(0,\infty)$, we have
\begin{align*}
\dot{g}_1(t)&=\alpha_\infty\left((1+t)^{2\alpha_\infty-1}-(1-t)^{2\alpha_\infty-1}-2t^{2\alpha_\infty-1}\right)\\
&\leq \alpha_\infty \left((1+t)^{2\alpha_\infty-1}-(1-t)^{2\alpha_\infty-1}-(2t)^{2\alpha_\infty-1}\right)\\
&=\alpha_\infty (2\alpha_\infty-1)\left(\int_{1-t}^{1+t} s^{2\alpha_\infty-2}ds-\int_{0}^{2t} s^{2\alpha_\infty-2}ds\right)<0.
\end{align*}
For $t>1$, by the convexity of
$t^{2\alpha_\infty-1}$,
$$\dot{g}_1(t)=\alpha_\infty\left((1+t)^{2\alpha_\infty-1}
+(t-1)^{2\alpha_\infty-1}-2t^{2\alpha_\infty-1}\right)<0.
$$
Additionally, direct calculation shows that $\lim_{t\to\infty} g_1(t)=0.$ Thus we have that $g_1(t)$ is strictly decreasing over $(0,\infty)$ with $g_1(0)=1$ and $\lim_{t\to\infty}g_1(t)=0$. Clearly, for any $0<\delta<1$,
$$\sup_{t>\delta}g_1(t)<1,$$
implying that (\ref{g(t)}) holds for $1/2<\alpha_\infty<1$. This completes the proofs. }\\

{\bf Acknowledgement}:
P. Liu was supported by the Swiss National Science Foundation Grant 200021-175752/1.

Research of K. Kosi\'nski was conducted under scientific
Grant No. 2014/12/S/ST1/00491 funded by National Science Centre.
\small\bibliography{biblioteczka}

\end{document}